\documentclass[11pt,a4paper]{amsart}
\usepackage{amsmath}
\usepackage{amsfonts}
\newtheorem{lemma}{Lemma}
\newtheorem{theorem}{Theorem}
\newtheorem{prop}{Proposition}
\newtheorem{corollary}{Corollary}

\newtheorem{remark}{Remark}
\numberwithin{lemma}{section}
\numberwithin{prop}{section}
\theoremstyle{definition}
\newtheorem{defn}{Definition}
\numberwithin{remark}{section}
\numberwithin{defn}{section}

\newcommand{\R}{\mathbb{R}}

\newcommand{\Z}{\mathbb{Z}}
\newcommand{\N}{\mathbb{N}}

\newcommand{\D}{\mathcal{D}}
\newcommand{\B}{\mathcal{B}}
\newcommand{\G}{\mathcal{G}}

\DeclareMathOperator{\diam}{diam}

\DeclareMathOperator{\dimH}{dim_{H}}
\DeclareMathOperator{\Cdim}{dim_{C}}
\DeclareMathOperator{\closure}{closure}
\allowdisplaybreaks[3]

\title[On the Diophantine properties of $\lambda$-expansions]{On the Diophantine properties of $\boldsymbol{\lambda}$-expansions}

\author{Tomas Persson}
\address{Tomas Persson\\ Centre for Mathematical Sciences\\ Lund
  University\\ Box 118\\ 22100 Lund\\ Sweden}
\email{tomasp@maths.lth.se}
\author{Henry WJ Reeve}
\address{Henry WJ Reeve\\Department of Mathematics\\ The University of
  Bristol\\ University Walk\\Clifton\\ Bristol\\BS8 1TW\\UK}
\email{henrywjreeve@gmail.com}

\begin{document}

\begin{abstract}
For $\lambda \in (\frac{1}{2}, 1)$ and $\alpha$, we consider sets of
numbers $x$ such that for infinitely many $n$, $x$ is $2^{-\alpha
  n}$-close to some $\sum_{i=1}^n \omega_i \lambda^i$, where $\omega_i
\in \{0,1\}$. These sets are in Falconer's intersection classes for
Hausdorff dimension $s$ for some $s$ such that $- \frac{1}{\alpha}
\frac{\log \lambda}{\log 2} \leq s \leq \frac{1}{\alpha}$. We show
that for almost all $\lambda \in (\frac{1}{2}, \frac{2}{3})$, the
upper bound of $s$ is optimal, but for a countable infinity of values
of $\lambda$ the lower bound is the best possible result.
\end{abstract}


\maketitle

\section{Introduction}

Diophantine approximation deals with the approximation of real numbers
by rationals. A classic example is the set $J(\alpha)$ of all
$\alpha$-well approximable numbers,
\[
  J(\alpha) = \{\, x \in \R : |x - p/q| < q^{-\alpha} \text{ for
    infinitely many } (p,q) \in \Z \times \N \,\}.
\]
Dirichlet showed that $J(\alpha) = \R$ for $\alpha = 2$ and Jarn\'{i}k
\cite{jarnik} and Besicovitch \cite{besicovitch} showed that the
Hausdorff dimension of $K(\alpha)$ is $2/\alpha$ for all $\alpha\geq
2$.

The sets $J(\alpha)$ belong to a family of sets with an interesting
large intersection property, first introduced by Falconer in
\cite{falconerold, falconernew}.  Falconer defined classes
$\mathcal{G}^s$ of $G_{\delta}$ subsets of $\R^n$ with the property
that any set in $\mathcal{G}^s$ has Hausdorff dimension at least $s$,
and any countable intersection of bi-Lipschitz images of sets from
$\mathcal{G}^s$, also belongs to $\mathcal{G}^s$. There are several
equivalent ways to characterise the sets in $\mathcal{G}^s$ (see
\cite{falconernew}).

Falconer showed that the set $J (\alpha)$ is in the class
$\mathcal{G}^{2/\alpha}$ \cite{falconerold}. This implies that any
countable intersection of $J (\alpha)$ with sets from
$\mathcal{G}^{2/\alpha}$ has Hausdorff dimension $2/\alpha$.

Real numbers are typically represented by some imperfect truncation of
their expansion to some given integer base. This motivates the
classification of numbers according to the accuracy of their finite
expansions by considering sets of the form,
\begin{multline*}
\B(\alpha) = \{ \, x \in \R : |x - p/2^n| < 2^{-\alpha n}
\text{ for infinitely many } (p,n) \in \Z \times \N \,\}.
\end{multline*}
For each $\alpha$ the set $\B (\alpha)$ is of Hausdorff dimension
$1/\alpha$. Moreover, each $\B(\alpha)$ belongs to the class
$\mathcal{G}^{1/\alpha}$. We note that $\B(\alpha)=\D(\alpha)+\Z$
where,
\begin{multline*}
\D(\alpha) = \biggl\{ x \in [0,1] : \bigg|x-\sum_{i=1}^n \omega_i
2^{-i}\bigg|<2^{-n\alpha} \\ \text{ for infinitely many }\omega \in
\{0,1\}^n, n \in \N \biggr\}.
\end{multline*}
For each $n \in \N$ we let $\D_n$ denote the set of all $n$-th level
dyadic sums
\begin{eqnarray*}
\D_n := \left\lbrace \sum_{i=1}^n \omega_i
2^{-i}: \omega \in \{0,1\}^n\right\rbrace.
\end{eqnarray*}
The fact that $\B(\alpha)$ belongs to the class
$\mathcal{G}^{1/\alpha}$ is essentially a consequence of the fact that
each $\D_n$ is evenly distributed in $[0,1]$. This motivates the
heuristic principle that if $\{\D_n\}_{n \in\N}$ were replaced by
some other family of suitably well distributed sets then we should
still obtain a set with large intersection properties.

Now take some $\lambda \in \left(\frac{1}{2},1\right)$. Just as every
number between zero and one may be written as a binary expansions, any
number $x \in [0, \lambda(1-\lambda)^{-1}]$ may be written in the
form,
\begin{equation*}
x=\sum_{i=1}^{\infty}\omega_i \lambda^i,
\end{equation*}
for some $(\omega_i)_{i=1}^{\infty} \in \{0,1\}^{\N}$. Following
Pollicott and Simon \cite{lambda expansions deleted digits} we refer
to $\left(\omega_i\right)_{i=1}^{\infty}$ as the $\lambda$-expansion
of $x$. In this paper we shall study the approximation of real numbers
by the finite truncations of their $\lambda$-expansions. Hence, we
study sets of the form
\begin{multline*}
W_\lambda (\alpha) = \biggl\{\, x \in [0, \lambda/(1-\lambda)] : \bigg|x -
\sum_{k=1}^n \omega_k \lambda^k \bigg| < 2^{-\alpha n},\\ \text{ for
  infinitely many } \omega \in \{0,1\}^n,\, n \in \N \,\biggr\}.
\end{multline*}
Since $W_\lambda (\alpha)$ is a subset of $[0,\lambda / (1 -
  \lambda)]$ it cannot belong to any class $\mathcal{G}^s$. Instead we
will consider the corresponding versions of the classes
$\mathcal{G}^s$ for subsets of an interval $I$, denoted by
$\mathcal{G}^s (I)$. It is natural to conjeture that for almost every
$\lambda$, $W_\lambda (\alpha)$ belongs to the set
$\mathcal{G}^{1/\alpha} (I)$. This conjecture is motivated by our
heuristic principle combined with results concerning the distribution
of the $n$-th level $\lambda$-sums,
\begin{eqnarray*}
\D_{\lambda}(n):= \left\lbrace \sum_{i=1}^n \omega_i
2^{-i}: \omega \in \{0,1\}^n\right\rbrace.
\end{eqnarray*}
This topic has attracted a great deal of interest since the time of
Erd\H{o}s \cite{erdos}. Erd\H{o}s studied a class of measures known as
\textit{infinite Bernoulli convolutions} formed by taking the
distributions of the random variable
\[
\sum_{k=1}^\infty \pm \lambda^k,
\]
for some $\lambda \in \left(\frac{1}{2},1\right)$, where in each term
$+$ and $-$ are chosen independently and with equal
probability. Erd\H{o}s proved the existence of an interval $(a,1)$ for
which the infinite Bernoulli convolution is absolutely continuous for
almost every $\lambda \in (a,1)$ \cite{erdos}. Erd\H{o}s also proved
the existence of a countable family of $\lambda$ for which the
corresponding infinite Bernoulli convolution is not absolutely
continuous.
Nonetheless it was conjectured that for almost every $\lambda \in
\left(\frac{1}{2},1\right)$ the corresponding infinite Bernoulli
convolution is absolutely continuous. In a breakthrough work of
Solomyak this conjecture was answered in the affirmative
\cite{solomyak}. This implies that for typical $\lambda$ the sums
$\D_{\lambda}(n)$ are fairly evenly distributed in the sense of
Lebesgue.

We shall show that for almost all $\lambda \in (\frac{1}{2},
\frac{2}{3})$, the set $W_\lambda (\alpha)$ belongs to the class
$\mathcal{G}^{\frac{1}{\alpha}} (I)$. However there is a dense set of
$\lambda$ such that the dimension of $W_{\lambda}(\alpha)$ drops below
$1/\alpha$. We also show that for any $\lambda \in (\frac{1}{2}, 1)$,
the set $W_\lambda (\alpha)$ belongs to $\mathcal{G}^s (I)$, at least
for $s = - \frac{1}{\alpha} \frac{\log \lambda}{\log 2}$. We also show
that this estimate is sharp in the sense that there exists a countable
set of $\lambda$ for which $\dimH W_\lambda (\alpha) = -
\frac{1}{\alpha} \frac{\log \lambda}{\log 2}$, and hence $W_\lambda
(\alpha)$ is not in the class $\mathcal{G}^s (I)$ for any $s$ larger
than $- \frac{1}{\alpha} \frac{\log \lambda}{\log 2}$.

\section{Notation and Statement of Results}

We begin by defining the classes $\mathcal{G}^s (I)$ referred to in
the introduction. One characterisation of Falconer's class
$\mathcal{G}^s$ is as follows \cite{falconernew}. $\mathcal{G}^s$ is
the set of all $G_{\delta}$ sets $A$ which have the property that for
any countable collection $\{f_j\}_{j \in \N}$ of similarity
transformations $f_j: \R \rightarrow \R$ we have,
\[
\dimH \Bigl(\bigcap_{j \in \N} f_j (A) \Bigl) \geq s.
\]

The class $\mathcal{G}^s (I)$ may be defined in terms of $\mathcal{G}^s$.

\begin{defn}
Given an interval $I$, the class $\G^s(I)$ is the class of subsets
of $I$ given by $\mathcal{G}^s(I):= \left\lbrace A \subseteq I:
A+\diam(I) \cdot \Z \in \mathcal{G}^s \right\rbrace.$
\end{defn}

We let $I_{\lambda}$ denote the closed interval
$[0,\lambda/(1-\lambda)]$ which consists of all points $x\in \R$ which
may be written in the form $x=\sum_{i=1}^{\infty}\omega_i\lambda^{i}$
for some $\omega\in \{0,1\}^{\N}$.  We shall consider the sets
$W_{\lambda}(\alpha)$ of points which are $\alpha$-well-approximated
by $\lambda$-expansions,
\begin{equation*}
W_{\lambda}(\alpha):= \bigcap_{m\in \N}\bigcup_{n\geq m}
\bigcup_{\omega \in \{0,1\}^n}\biggl\{\, x \in
I_\lambda : \bigg|x-\sum_{i=1}^n \omega_i
\lambda^{i}\bigg|<2^{-n\alpha} \,\biggr\}.
\end{equation*}

\begin{theorem}\label{main bullet points}
  Choose $\alpha \in (1, \infty)$. 
  \begin{enumerate}
  \item For all $\lambda \in \left(\frac{1}{2},1\right)$, $\dim
    W_{\lambda}(\alpha)\leq \frac{1}{\alpha}$,
  \item For almost every $\lambda \in
    \left(\frac{1}{2},\frac{2}{3}\right)$, $W_{\lambda}(\alpha) \in
    \G^{s}(I_{\lambda})$ for $s=\frac{1}{\alpha}$,
  \item For a dense set of $\lambda \in \left(\frac{1}{2},1\right)$,
    $\dim W_{\lambda}(\alpha) < \frac{1}{\alpha}$,
  \item For all $\lambda \in \left(\frac{1}{2},1\right)$,
    $W_{\lambda}(\alpha) \in \G^{s}(I_{\lambda})$ for $s=-\frac{\log
    \lambda}{\log 2}\frac{1}{\alpha}$,
  \item For a countable set of $\lambda \in
    \left(\frac{1}{2},1\right)$, $\dim W_{\lambda}(\alpha)=-\frac{\log
      \lambda}{\log 2}\frac{1}{\alpha}$.
  \end{enumerate}
\end{theorem}

In addition to Theorem \ref{main bullet points} (2) we also have the
following upper bound on the dimension of the set of exceptions.

\begin{theorem}\label{theorem typical}
  Given $\alpha >1$ and $s\leq \frac{1}{\alpha}$ we have,
  \begin{equation*}
    \dimH \Bigl\{\, \lambda \in
    \Bigl( \frac{1}{2},\frac{2}{3} \Bigr): W_{\lambda}(\alpha) \notin
    \G^{s}(I_{\lambda}) \,\Bigr\} \leq s.
  \end{equation*}
\end{theorem}
The remainder of the paper is structured as follows. In Section
\ref{s P of Th 2} we prove Theorem \ref{theorem typical}, which
implies Theorem \ref{main bullet points} (2). In Section
\ref{sec:lowerbound} we establish the uniform lower bound given in
Theorem \ref{main bullet points} (4). In Section \ref{s covering args}
we prove the upper bounds in Theorem \ref{main bullet points} parts
(1), (3) and (5).

\section{Proof of Theorem \ref{theorem typical}}\label{s P of Th 2}
In this section we prove Theorem \ref{theorem typical}. The proof of
this theorem is influenced by Rams' work on the dimension of the exceptional set for families of self-similar measures with overlaps \cite{rams}.

For each $\lambda\in\left(\frac{1}{2},\frac{2}{3}\right)$ and $k\in
\N\cup\{0\}$, $r\in \N$ we define a pair of proximity numbers
\begin{align*}
\tilde{P}_n(\lambda,k,r) &:= \# \biggl\{ (\omega,\kappa)\in
\left(\{0,1\}\right)^n:
\bigg|\sum_{i=1}^n(\omega_i-\kappa_i)\lambda^i\bigg|\leq r \cdot
\lambda^{n+k} \biggr\},\\ P_n(\lambda,k,r) &:= \# \biggl\{
(\omega,\kappa)\in \left(\{0,1\}\right)^n:
\bigg|\sum_{i=1}^n(\omega_i-\kappa_i)\lambda^i\bigg| \\ & \hspace{5cm}
\leq r \cdot \lambda^{n+k}\text{ and }\omega_1\neq \kappa_1 \biggr\}.
\end{align*}

\begin{lemma}\label{proximity inequality}
For all $n \in \N$ and $k\in \N\cup\{0\}$, $r\in \N$ we have,
\[
\tilde{P}_n(\lambda,k,r)\leq
2^n+\sum_{l=1}^{n}2^{n-l}P_l(\lambda,k,r).
\]
\end{lemma}

\begin{proof}
  For notational convenience we let,
  \begin{eqnarray*}
    \mathcal{P}_n(\lambda,k,r):=\left\lbrace (\omega,\kappa)\in
    \left(\{0,1\}\right)^n:
    \bigg|\sum_{i=1}^n(\omega_i-\kappa_i)\lambda^i\bigg|\leq r \cdot
    \lambda^{n+k} \right\rbrace.
  \end{eqnarray*}
  We begin by writing,
  \begin{align}\label{ summands for different proximity numbers }
    \tilde{P}_n (\lambda,k,r) = \#\bigl\{ (\omega,\kappa)\in
    \mathcal{P}_n(\lambda,k,r) & : \omega = \kappa \bigr\} \\ +
    \sum_{l=1}^{n}\#\bigl\{ (\omega,\kappa)\in
    \mathcal{P}_n(\lambda,k,r) & : \omega_i=\kappa_i \text{ for }i\leq
    n-l \nonumber \\ & \phantom{:} \text{ and }\omega_{n-l+1}\neq
    \kappa_{n-l+1}\bigr\}. \nonumber
  \end{align}
  The cardinality of the first summand is clearly equal to
  $2^n$. Given a pair $(\omega,\kappa)\in \mathcal{P}_n(\lambda)$ with
  $\omega_i=\kappa_i$ for $i\leq n-l$ and $\omega_{n-l+1}\neq
  \kappa_{n-l+1}$, for some $l \in \{1,\cdots,k\}$ there exists some
  $\eta \in \{0,1\}^{n-l}$ and $\zeta,\xi \in \{0,1\}^l$ with
  $\eta_1\neq \zeta_1$ such that $\omega= \eta \zeta$ and $\kappa =
  \eta \xi$. It follows from the fact that $(\omega,\kappa)\in
  \mathcal{P}_n(\lambda,k,r)$ that,
  \begin{equation*}
    \lambda^{n-l} \bigg|\sum_{i=1}^l(\zeta_i-\xi_i)\lambda^i\bigg|\leq
    r\cdot \lambda^{n+k}.
  \end{equation*}
  Thus,
  \begin{equation*}
    \bigg|\sum_{i=1}^l(\zeta_i-\xi_i)\lambda^i\bigg|\leq r \cdot
    \lambda^{l+k}.
  \end{equation*}
  It follows that the number of elements of
  \begin{equation*}
    \left\lbrace (\omega,\kappa)\in \mathcal{P}_n(\lambda,k,r):
    \omega_i=\kappa_i \text{ for }i\leq n-l \text{ and
    }\omega_{n-l+1}\neq \kappa_{n-l+1}\right\rbrace
  \end{equation*}
  is equal to $P_l(\lambda,k,r)$ multiplied by the number of possible
  initial strings of length $n-l$.  Thus,
  \begin{multline*}
    \#\left\lbrace (\omega,\kappa)\in \mathcal{P}_n(\lambda,k,r):
    \omega_i=\kappa_i \text{ for }i\leq n-l \text{ and
    }\omega_{n-l+1}\neq \kappa_{n-l+1}\right\rbrace \\ =
    2^{n-l}P_l(\lambda,k,r).
  \end{multline*}
  Substituting into equation (\ref{ summands for different proximity
    numbers }) completes the proof of the Lemma.
\end{proof}

To prove that $W_\lambda (\alpha)$ is in $\mathcal{G}^s (I_\lambda)$
we will need good estimates on the numbers $P_n (\lambda, k, r)$. We
will get such estimates for almost all $\lambda \in (\frac{1}{2},
\frac{2}{3})$, and the first step to get this is using the following
lemma.

\begin{lemma}[Shmerkin, Solomyak \cite{shmerkinsolomyak}] \label{Shmerkin Solomyak}
  For any $\varepsilon > 0$, there exists some $\delta>0$ such that
  for all polynomials of the form $g(\lambda)=1+\sum_{i=1}^n
  a_i\lambda^i$ with $a_i \in \{-1,0,1\}$ and for all $\lambda\in
  (\frac{1}{2},\frac{2}{3} - \varepsilon)$ we have
  $g'(\lambda)<-\delta$ whenever $g(\lambda)<\delta$.
\end{lemma}

Given $n \in \N$, a pair $(\omega,\kappa)\in (\{0,1\}^n)^2$
and $\gamma>0$ we let
\begin{equation*}
  I_n(\omega,\kappa,\gamma):=\left\lbrace \lambda \in
  \left(\frac{1}{2},\frac{2}{3}\right):
  \bigg|\sum_{i=1}^n(\omega_i-\kappa_i)\lambda^i\bigg|\leq \gamma
  \right\rbrace.
\end{equation*}

\begin{lemma}\label{cor to Sh and Sol}
  Let $\delta$ be as in Lemma \ref{Shmerkin Solomyak}. Then for all
  $\gamma\in (0,\delta/2)$ and all pairs $(\omega,\kappa)\in
  (\{0,1\}^n)^2$ with $\omega_1 \neq \kappa_1$,
  $I_n(\omega,\kappa,\gamma)$ has diameter not exceeding
  $4\delta^{-1}\gamma$.
\end{lemma}

\begin{proof} 
  Since $\omega_1\neq \kappa_1$ we may assume without loss of
  generality that $\omega_1=1$ and $\kappa_1=0$. Choose $\gamma\in
  (0,\delta/2)$ and all pairs $(\omega,\kappa)\in
  \left(\{0,1\}\right)^n$ with $\omega_1 \neq \kappa_1$. Now let
  $g(\lambda):= \sum_{i=1}^{n}(\omega_i-\kappa_i)\lambda^{i-1}$, which
  is of the required form for Lemma \ref{Shmerkin Solomyak}. We note
  that $\lambda \in I_n(\omega,\kappa,\gamma)$ implies
  $|g(\lambda)|<\gamma/\lambda <\delta$. By Lemma \ref{Shmerkin
    Solomyak} $g'(\lambda)<-\delta$ whenever
  $g(\lambda)<\delta$. Suppose $I_n(\omega,\kappa,\gamma)\neq
  \emptyset$ and choose $\lambda_0:=\inf
  I_n(\omega,\kappa,\gamma)$. It follows from Rolle's theorem that for
  all $\lambda\geq \lambda_0$, $g(\lambda) \leq \gamma<\delta$ and
  hence, $g'(\lambda)<-\delta$. Hence,
  $I_n(\omega,\kappa,\gamma)\subseteq
  [\lambda_0,\lambda_0+4\delta^{-1}\gamma]$.
\end{proof}

Using the following result by Rams \cite{rams} we will prove our
desired estimates for the numbers $P_n (\lambda, k, r)$.

\begin{lemma}[Rams \cite{rams}] \label{Ram's combinatorial lemma}
  Suppose we have a family of sets $\{E_i\}_i$ with $E_i$ of diameter
  $d_i$. Let $\rho>0$ be some positive real number and $b\in
  \N$. Then, the set of points which belong to at least $b$ of the
  sets $E_i$ may be covered by some family of intervals
  $\{\tilde{E}_j\}_j$ so that $\tilde{E}_j$ has diameter $\tilde{d}_j$
  with $\sup_j \tilde{d}_j \leq 4 \sup_i d_i$ and
\begin{equation*}
  \sum_j \tilde{d}_j^{\rho} \leq 4^{\rho} \cdot \frac{1}{b} \sum_i
  d_i^{\rho}.
\end{equation*}
\end{lemma}

For each $s$ we shall let
\begin{equation*}
  A(s):= \bigcup_{r\in \N}\bigcap_{m\in \N}\bigcup_{n \geq m}
  \bigcup_{k\geq 0}\left\lbrace \lambda \in
  \left(\frac{1}{2},\frac{2}{3}\right):
  P_n(\lambda,k,r)>4^n\lambda^{s(n+k)} \right\rbrace.
\end{equation*}
\begin{lemma}
  For all $s \in (0,1)$ we have $\dimH A(s)\leq s$.
\end{lemma}
\begin{proof}
  Choose $\rho > s$ and take some $r\in \N$. Take $n\in \N$ with
  $\lambda^{n}<\delta/2$. Note that each $\lambda \in
  \left(\frac{1}{2},\frac{2}{3}\right)$ with
  $P_n(\lambda,k,r)>4^n\lambda^{s(n+k)}$ is contained within
  $I_n(\omega,\kappa,r\cdot\lambda^{n+k})$ for at least $\lceil
  4^n\lambda^{s(n+k)} \rceil$ pairs $(\omega,\kappa)\in
  \left(\{0,1\}^n\right)^2$. Now by Lemma \ref{cor to Sh and Sol} each
  $I_n(\omega, \kappa,r\cdot\lambda^{n+k})$ has diameter not exceeding
  $4\delta^{-1}r\lambda^{n+k}$. Thus, by Lemma \ref{Ram's
    combinatorial lemma} we may cover
  \[
  \left\lbrace \lambda \in \left(\frac{1}{2},\frac{2}{3}\right):
  P_n(\lambda,k,r)>4^n\lambda^{s(n+k)} \right\rbrace
  \]
  with a family of sets $A_{i}^n(s,k)$ of diameter no greater than
  $16r\delta^{-1}\lambda^{n+k}$ and satisfying,
  \begin{align*}
    \sum_{i} & \diam(A_i^n(s,k))^{\rho} \leq 4^{\rho}\cdot
    (4^{-n}\lambda^{-s(n+k)}) \cdot \\ & \hspace{4cm} \cdot
    \sum_{(\omega,\kappa)\in \left(\{0,1\}\right)^n}
    \diam\left(I_n(\omega,\kappa,r\lambda^{n+k})\right)^{\rho}
    \\ &\leq (4r)^{\rho}\cdot (4^{-n}\lambda^{-s(n+k)}) \cdot 2^{2n}
    \cdot \left(4\delta^{-1}\lambda^{n+k}\right)^{\rho}\\ &\leq
    (16r/\delta)^{\rho} \lambda^{(n+k)(\rho-s)}.
  \end{align*}

Consequently, we may cover
\[
\bigcup_{k\geq 0}\left\lbrace \lambda \in
\left(\frac{1}{2},\frac{2}{3}\right):
P_n(\lambda,k,r)>4^n\lambda^{s(n+k)} \right\rbrace
\]
with sets $A_i^n(s,k)$ of diameter no greater than
$16r\delta^{-1}\lambda^n$ and satisfying,
\begin{eqnarray*}
  \sum_k\sum_{i}\diam(A_i^n(s,k))^{\rho} &\leq &
  (16r/\delta)^{\rho}\left(1-\lambda^{\rho-s}\right)^{-1}
  \lambda^{n(\rho-s)}.
\end{eqnarray*} 
It follows that for each $m\in \N$,
\[
\bigcup_{n\geq m}\bigcup_{k\geq 0}\left\lbrace \lambda \in
\left(\frac{1}{2},\frac{2}{3}\right):
P_n(\lambda,k,r)>4^n\lambda^{s(n+k)} \right\rbrace
\]
may be covered by a family of sets $\bigcup_{n\geq m}
\bigcup_k\bigcup_i A_i^n(s,k)$ of diameter not exceeding
$16r\delta^{-1}\lambda^{m}$ with
\begin{eqnarray*}
  \sum_{n\geq m} \sum_k\sum_{i}\diam(A_i^n(s,k))^{\rho} &\leq &
  (16r/\delta)^{\rho}\left(1-\lambda^{\rho-s}\right)^{-1} \cdot
  \sum_{n\geq m} \lambda^{n(\rho-s)} \\ &\leq &
  (16r/\delta)^{\rho}\left(1-\lambda^{\rho-s}\right)^{-2}
  \lambda^{m(\rho-s)}.
\end{eqnarray*}
For every $m\in \N$ we have,
\begin{align*}
  \bigcap_{m \in \N} & \bigcup_{n \geq m} \bigcup_{k\geq 0}\left\lbrace
  \lambda \in \left(\frac{1}{2},\frac{2}{3}\right):
  P_n(\lambda,k,r)>4^n\lambda^{s(n+k)} \right\rbrace \\ \subseteq
  & \bigcup_{n\geq m}\bigcup_{k\geq 0}\left\lbrace \lambda \in
  \left(\frac{1}{2},\frac{2}{3}\right):
  P_n(\lambda,k,r)>4^n\lambda^{s(n+k)} \right\rbrace.
\end{align*}
Thus,
\begin{equation*}
  \dimH\left(\bigcap_{m\in \N}\bigcup_{n \geq m} \bigcup_{k\geq
    0}\left\lbrace \lambda \in \left(\frac{1}{2},\frac{2}{3}\right):
  P_n(\lambda,k,r)>4^n\lambda^{s(n+k)} \right\rbrace\right) \leq \rho.
\end{equation*}
$A(s)$ is a countable union of such sets and so $\dimH A(s)\leq
\rho$. Since $\rho>s$ was arbitrary the Lemma holds.
\end{proof}

Let $\mathcal{D} = \{0,1\}$. For a natural number $n$ we denote by
$\mathcal{D}^n$ the set of words $(\omega_1, \omega_2, \ldots,
\omega_n)$ of length $n$ such that each $\omega_k$ is in
$\mathcal{D}$. Similarly we denote the set of all such infinite
sequences by $\Sigma$. If $\omega$ is an element of $\Sigma$ or
$\mathcal{D}^l$ with $l \geq n$, then we let $\omega | n$ denote the
element in $\mathcal{D}^n$ such that $\omega$ and $\omega | n$ are
equal on the first $n$ places.  Given an $\omega \in \mathcal{D}^{n}$
we define a function $g_\omega (x) = \sum_{i=1}^{n} \omega_i \lambda^i
+ \lambda^{n} x$.

\begin{lemma}\label{lots of cylinders} 
  Given a similarity map $f \colon \R \rightarrow \R$ defined by $f
  \colon x \mapsto rx+t$ for some fixed $r, t \in \R$, together with
  any closed interval $A$ with non-empty interior and $\diam(A) < r
  \diam\left(I_{\lambda}\right)$ there exists an integer $n(A,f) \in
  \Z$ and a finite string $\omega=\omega(A,f)\in \D^{\theta}$, with
  length $\theta$ depending only on the magnitude of the derivative
  $|f'|$ and the diameter $\diam(A)$ of $A$, such that the interval
  $f\left( g_{\omega}(I_{\lambda})+ n(A,f)\cdot
  \diam(I_{\lambda})\right)$ is contained within $A$ and has diameter
  at least $\lambda/4\cdot \diam(A)$.
\end{lemma}

\begin{proof}
  Since $\diam(A) < r \diam\left(I_{\lambda}\right)$,
  $\diam(f^{-1}(A)) < \diam\left(I_{\lambda}\right)$. Hence, the
  closed interval $f^{-1}(A)$ intersects at most two of the intervals
  \[
  \left\lbrace I_{\lambda}+n \diam(I_{\lambda})\right\rbrace_{n \in \Z}.
  \]
  As such, we may choose $n(A,f) \in \Z$ so that,
  \[
  \diam \left(f^{-1}(A) \cap \left(I_{\lambda}+n(A,f)
  \diam(I_{\lambda})\right)\right) \geq \frac{1}{2} \cdot
  \diam\left(f^{-1}(A)\right).
  \]
  Equivalently, $\diam \left(Z\right) \geq \frac{1}{2} \cdot
  \diam\left(f^{-1}(A)\right)$ where
  \[
  Z=\left(f^{-1}(A)-n(A,f) \diam(I_{\lambda}) \right) \cap
  I_{\lambda}.
  \]
  Let $x$ denote the midpoint of $Z$. Since $x\in I_{\lambda}$ we may
  write $x = \sum_{i=1}^{\infty}\omega_i\lambda^i =
  \bigcap_{n\in\N}g_{\omega|n}(I_{\lambda})$. We choose $\theta$ so
  that
  \begin{equation*}
    \theta:= \bigg\lfloor
    \frac{\log\left((1-\lambda)\diam(A)/4|f'|\right)}{\log
      \lambda}\bigg\rfloor.
  \end{equation*}
  In particular, $\theta$ depends only upon the magnitude of the
  derivative $|f'|$ and the diameter $\diam(A)$ of $A$. Since $f$ is a
  similarity and $I_{\lambda}$ is of diameter $\lambda/(1-\lambda)$,
  it follows that
  \begin{eqnarray*}
    \diam\left(g_{\omega|\theta}(I_{\lambda})\right) &=&
    \frac{\lambda^{\theta+1}}{1-\lambda}\\ &<&
    \frac{\diam(A)}{4r}\\ &=& \frac{1}{2}\cdot
    \frac{\diam(f^{-1}(A))}{2}\\ &<&\frac{1}{2}\cdot \diam(Z).
  \end{eqnarray*}
  Since $x$ is the midpoint of $Z$ and
  $g_{\omega|\theta}(I_{\lambda})$ contains $x$ we have
  \[
  g_{\omega|\theta}(I_{\lambda}) \subseteq Z \subseteq
  f^{-1}(A)-n(A,f) \cdot \diam(I_{\lambda}).
  \]
  Hence,
  \[
  f\left(g_{\omega|\theta}(I_{\lambda})+n(A,f) \cdot
  \diam(I_{\lambda})\right) \subseteq A.
  \]
  Moreover,
  \begin{eqnarray*}
    \diam\left(g_{\omega|\theta}(I_{\lambda})\right) &=&
    \frac{\lambda^{\theta+1}}{1-\lambda}\\ &\geq& \lambda \cdot
    \frac{\diam(A)}{4r}.
  \end{eqnarray*}
  Thus,
  \begin{equation*}
    \diam\left(f\left(g_{\omega|\theta}(I_{\lambda})+n(A,f) \cdot
    \diam(I_{\lambda})\right) \right)\geq \frac{\lambda}{4}\cdot
    \diam(A). \qedhere
  \end{equation*}
\end{proof}

Given a positive number $r>0$ and a finite set $\Omega$ and two
functions $\varphi_1, \varphi_2: \Omega \rightarrow \R$ we shall let
\begin{equation*}
  N_r(\varphi_1,\varphi_2):=\#\left\lbrace (x,y)\in \Omega^2 :
  |\varphi_1(x)-\varphi_2(y)|\leq r\right\rbrace.
\end{equation*}

\begin{lemma}\label{translation lemma}
Given $r > 0$, any finite set $\Omega$, any function $\varphi \colon
\Omega \rightarrow \R$ and any $t \in \R$, we have $N_r
(\varphi,\varphi+t) \leq 4 \cdot N_r (\varphi,\varphi)$.
\end{lemma}

\begin{proof}[Proof of Lemma \ref{translation lemma}]
  Since the inequality $|\varphi_1(x)-\varphi_2(y)|\leq r$ holds if
  and only if $|\varphi_1(x)/r-\varphi_2(y)/r| \leq 1$, it is
  sufficient to prove the lemma in the case $r=1$.

  For each $n \in \Z$ we let $a_{n}:= \# \left(\Omega \cap
  \varphi^{-1}[n,n+1)\right)$. Given any pair $(a,b) \in \Omega^2$
    with $\varphi(a),\varphi(b) \in [n, n+1)$ for some $n \in \Z$ we
      have $|\varphi(a)-\varphi(b)| \leq 1$. For each $n\in \Z$ there
      are at least $a_n^2$ such pairs, so $N_1 (\varphi,\varphi) \geq
      \sum_{n \in \Z} a_n^2$.

  Now suppose $a, b \in \Omega$, $\varphi(a) \in [n, n+1)$,
    $|\varphi(a)-(\varphi(b)+t)| \leq 1$. Since $n \leq \varphi(a)
    <n+1$, so $n-1 \leq \varphi(b)+t < n+2$, and so
  \[
  n-(\lceil t\rceil+1) \leq n-1-t \leq \varphi(b) < n+2-t<(n-(\lfloor
  t \rfloor-1)) + 1.
  \]
  Hence, $\varphi(b)$ is in $[n-p, n-p+1)$ for some integer $p$ with
    $\lfloor t \rfloor-1 \leq p \leq \lceil t\rceil+1$. Thus, for each
    $a \in \Omega$ with $\varphi(a) \in [n, n+1)$ we have
  \begin{align*}
    \#\{\, b \in \Omega &: |\varphi(a)-(\varphi(b)+t)| \leq
    1 \,\} \\
    &\leq \sum_{\lfloor t \rfloor-1 \leq p \leq \lceil t\rceil+1} \#
    \left(\Omega \cap \varphi^{-1}[n-p,n-p+1)\right) \\ &=
      \sum_{\lfloor t \rfloor-1 \leq p \leq \lceil t\rceil+1} a_{n-p}.
  \end{align*}
  Thus, for each $n \in \N$,
  \begin{multline*}
    \#\left\lbrace (a,b) \in \Omega^2: \varphi(a) \in [n, n+1),
      |\varphi(a)-(\varphi(b)+t)| \leq 1 \right\rbrace \\ \leq
      \sum_{\lfloor t \rfloor-1 \leq p \leq \lceil t\rceil+1} a_n
      \cdot a_{n-p}.
  \end{multline*}
  Hence, $N_1 (\varphi, \varphi+t)\leq \sum_{n \in \Z} \sum_{\lfloor t
    \rfloor-1 \leq p \leq \lceil t\rceil+1}a_n a_{n-p}$.  Thus, by
  Cauchy--Schwarz we have,
  \begin{align*}
    N_1 (\varphi, \varphi+t)& \leq \sum_{\lfloor t \rfloor-1 \leq p \leq
      \lceil t\rceil+1}\sum_{n \in \Z} a_n a_{n-p}\\ & \leq
    \sum_{\lfloor t \rfloor-1 \leq p \leq \lceil
      t\rceil+1}\sqrt{\sum_{n \in \Z} a_n^2 \cdot \sum_{n \in \Z}
      a_{n-p}^2}\\ & \leq \sum_{\lfloor t \rfloor-1 \leq p \leq \lceil
      t\rceil+1}\sum_{n \in \N} a_n^2 \\ & \leq 4
    N_1 (\varphi,\varphi). \qedhere
  \end{align*}
\end{proof}
\begin{remark}
It is natural to ask whether or not $4$ is the optimal constant possible in Lemma \ref{translation lemma}. Matthew Aldridge has provided an inductive demonstration that $N_r(\varphi,\varphi+t) < 2 \cdot N_r (\varphi,\varphi)$, whilst Oliver Roche-Newton has produced a family of counterexamples showing that such a bound is optimal.
\end{remark}

\begin{lemma}\label{Constant depending on r}
  Suppose $\lambda \notin A(s)$ and $r\in \N$. Then there exists a
  constant $C(r)>0$, such that for all $n \in \N$
  and all $k\in \N\cup\{0\}$,
  \begin{equation*}
    \tilde{P}_n(\lambda,k,r)\leq C(r) \cdot 2^n+4^{n}n\lambda^{s(n+k)}.
  \end{equation*}
\end{lemma}

\begin{proof}
  Suppose $\lambda\notin A(s)$ and $r\in \N$. Then there exists some
  $N_0\in \N$ such that for all $n\geq N_0$ and all $k\in \N\cup
  \{0\}$, $P_n(\lambda,k,r) \leq 4^n\lambda^{s(n+k)}$. Thus, if we
  take $C:=1+\sum_{l=1}^{N_0}2^{-l}P_l(\lambda,0,r)$ then by Lemma
  \ref{proximity inequality} then we have,
  \begin{eqnarray*}
    \tilde{P}_n(\lambda,k,r)&\leq&
    2^n+\sum_{l=1}^{n}2^{n-l}P_l(\lambda,k,r)\\ &\leq&
    2^n\left(1+\sum_{l=1}^{N_0}2^{-l}P_l(\lambda,k,r)\right)+
    \lambda^{sk} \cdot \sum_{l=N_0+1}^n 2^{n-l}\cdot
    (4\lambda^s)^l\\ &\leq&
    2^n\left(1+\sum_{l=1}^{N_0}2^{-l}P_l(\lambda,0,r)\right)+
    \lambda^{sk} \cdot \sum_{l=N_0+1}^n (4\lambda^s)^n\\ &\leq& C\cdot
    2^n+ 4^{n}n\lambda^{s(n+k)},
  \end{eqnarray*}
  where we used the fact that $\lambda \geq \frac{1}{2}$, so $4
  \lambda^s\geq 2$.
\end{proof}

\begin{prop}
  Suppose $\lambda \notin A(s)$ for some $s\leq
  \frac{1}{\alpha}$. Then $W_{\lambda}(\alpha) \in \G^s(I_{\lambda})$.
\end{prop}

\begin{proof}
  To prove the proposition we begin by fixing $\lambda\notin A(s)$,
  $\alpha>1$ and a sequence of similarity maps $\{f_j\}_{j\in \N}$. We
  shall show that
  \begin{equation*}
    \dimH\left( \bigcap_{j\in\N}
    f_j\left(W_{\lambda}(\alpha)+\diam(I_{\lambda})\cdot \Z \right)\right)\geq s.
  \end{equation*}
  To do so we shall construct a subset $\Lambda \subset
  \bigcap_{j\in\N} f_j\left(W_{\lambda}(\alpha)+\diam(I_{\lambda})\cdot \Z \right)$ supporting a
  measure $\nu$ with correlation dimension $s$. Without loss of generality we may assume that $f_1:x \mapsto 2x$. We begin by choosing a sequence of natural numbers $(j(q))_{q\in \N \cup\{0\}}$ so that $j(0)=1$ and for each
  $k\in \N$,
  \begin{equation}\label{Infinite k}
    \#\{\, q : j(q) = k \,\} = \infty.
  \end{equation}

  Let $\Sigma_* = \{\emptyset\} \cup_n \mathcal{D}^n$.  We shall
  recursively construct sequences of integers $(\gamma_q)_{q \in \N}$,
  $(\hat{\gamma}_q)_{q \in \N}$, $(\theta_q)_{q \in \N}$ and
  $(m_q)_{q\in \N}$ along with closed intervals
  $(\Delta_{\omega})_{\omega\in \Sigma_*}$ and
  $(\hat{\Delta}_{\omega})_{\omega\in \Sigma_*}$ and positive reals
  $(\delta_n)_{n\in \N\cup\{0\}}$, $(\hat{\delta}_n)_{n\in
    \N\cup\{0\}}$ with the property that for any $\omega \in \Sigma_*$
  and $\eta \in \D$,
  \begin{equation*}
    \Delta_{\omega}\supseteq \hat{\Delta}_{\omega}\supseteq
    \Delta_{\omega\eta}.
  \end{equation*}
  Moreover, given any word $\omega \in \D^n$ for some $n \in \N\cup
  \{0\}$ we have $\diam(\Delta_{\omega})=\delta_n$ and
  $\diam(\hat{\Delta}_{\omega})=\hat{\delta}_n$. We also have
  $\hat{\delta}_n\leq \delta_n \leq \lambda^{n+1}/(1-\lambda)$ for all
  $n\in \N\cup\{0\}$. In addition, $\lambda^{\gamma_q} < \lVert
  f_{j(q)}' \rVert_{\infty}$ for $q \geq 1$.

  First let $\gamma_0=\hat{\gamma}_0=\theta_0=m_0=0$,
  $\Delta_{\emptyset}=\hat{\Delta}_{\emptyset}=I_{\lambda}$ and
  $\delta_0=\hat{\delta}_0=\lambda/(1-\lambda)$.

  Suppose we have chosen $\gamma_l$, $\theta_l$ and $m_l$ for $l\leq
  q$ and for all $n \leq \Gamma (q) := \sum_{l\leq q}\gamma_l$ we have
  defined $\delta_n$, $\hat{\delta}_n$ and for $\omega\in \D^n$ we
  have $\Delta_{\omega}$ and $\hat{\Delta}_{\omega}$, all satisfying
  the required properties.

  For the inductive step we first apply Lemma \ref{lots of cylinders}
  to obtain $\left(\omega(\kappa)\right)_{\kappa \in
    \D^{\Gamma (q)}}$ and $\left(n(\kappa)\right)_{\kappa \in
    \D^{\Gamma (q)}}$ with $n(\kappa)=n  (\hat{\Delta}_{\kappa},f_{j(q)}) \in \Z$ and $\omega(\kappa)=\omega
  (\hat{\Delta}_{\kappa},f_{j(q)}) \in \Sigma_*$ for each $\kappa \in
  \D^{\Gamma (q)}$ so that,
  \begin{itemize}
  \item[(1)] $f_{j(q)}\left( g_{\omega(\kappa)}(I_{\lambda})+n(\kappa)\diam(I_{\lambda})\right) \subseteq
    \hat{\Delta}_{\kappa}$,
  \item[(2)] $\diam\left(f_{j(q)}\left(g_{\omega(\kappa)}(I_{\lambda})+n(\kappa)\diam(I_{\lambda})\right)\right) \geq \frac{\lambda}{4}\cdot
    \diam (\hat{\Delta}_{\kappa})$.
  \end{itemize} 

  By supposition, $\diam (\hat{\Delta}_{\kappa}) = \delta_{\Gamma (q)}$ for all $\kappa \in \D^{\Gamma (q)}$. Consequently,
  by Lemma \ref{lots of cylinders} the length of $|\omega(\kappa)|$ is
  uniform over all $\kappa \in \D^{\Gamma (q)}$. We denote
  this uniform length by $\theta_{q+1}$.

  Choose $\gamma_{q+1},\hat{\gamma}_{q+1} \in \N$ so that,
  \begin{eqnarray*}
    \gamma_{q+1}&>& q\gamma_q\theta_{q+1} \cdot (-\log
    \delta_{\Gamma (q)}),\\
    \gamma_{q+1}&>& \frac{\log |f'_{j(q+1)}|}{\log \lambda}
    ,\\ \hat{\gamma}_{q+1}&=&\gamma_{q+1}+\theta_{q+1}.
  \end{eqnarray*}
  and let
  \begin{equation*}
    m_{q+1}:=\bigg\lfloor \left(\frac{\log 2^{-\alpha}}{\log
      \lambda}-1\right)\hat{\gamma}_{q+1} -\frac{\log(1-\lambda)}{\log \lambda}\bigg\rfloor+1,
  \end{equation*}
  so that
  \begin{equation}\label{m q inequalities}
    \lambda^{\hat{\gamma}_{q+1}+m_{q+1}}/(1-\lambda)< 2^{-\alpha
      \hat{\gamma}_{q+1}} \leq \lambda^{\hat{\gamma}_{q+1}+m_{q+1}-1}/(1-\lambda).
  \end{equation}

  Given $\kappa \in \D^{\Gamma (q)}$ and $\tau\in \D^{l}$
  for some $l\leq \gamma_{q+1}$ we define
  \begin{equation*}
    \Delta_{\kappa \tau}:=f_{j(q)}\left(g_{\omega(\kappa)}\circ
    g_{\tau}(I_{\lambda})+n(\kappa)\cdot \diam(I_{\lambda})\right).
  \end{equation*}
  Thus, for all $\omega\in \D^{\Gamma (q) + l}$ for some
  $l\leq \gamma_{q+1}$ we have,
  \begin{eqnarray*}
    \diam\left(\Delta_{\omega}\right)= \delta_{\Gamma (q) + l} :=
    |f_{j(q)}'|\cdot \lambda^{\theta_{q+1}+l+1}/(1-\lambda).
  \end{eqnarray*}
  Moreover, for $l<\gamma_{q+1}$ we let $\hat{\Delta}_{\kappa
    \tau}:=\Delta_{\kappa \tau}$ and for $l=\gamma_{q+1}$,
  \begin{equation*}
    \hat{\Delta}_{\kappa \tau}:=f_{j(q)}\left( g_{\omega(\kappa)}\circ
    g_{\tau}\circ (g_0)^{m_{q+1}}(I_{\lambda})+n(\kappa) \cdot \diam(I_{\lambda})\right).
  \end{equation*}
  Hence, for all $\omega\in \D^{\Gamma (q) + l}$ for some
  $l<\gamma_{q+1}$ we have,
  \begin{eqnarray*}
    \diam (\hat{\Delta}_{\omega}) = \hat{\delta}_{\Gamma (q) + l} :=
    \delta_{\Gamma (q) + l},
  \end{eqnarray*}
  and for $\omega\in \D^{\Gamma (q+1)}$
  \begin{eqnarray*}
    \diam (\hat{\Delta}_{\omega}) = \hat{\delta}_{\Gamma
      (q+1)} := |f_{j(q)}'|\cdot
    \lambda^{\theta_{q+1}+\gamma_{q+1}+m_{q+1}+1}/(1-\lambda).
  \end{eqnarray*}

  It follows that for all $\eta \in \D^{\Gamma (q+1)}$
  \begin{multline} \label{Good approximant prop}
    \hat{\Delta}_{\eta}\subseteq f_{j(q)} \Biggl( \bigcup_{\omega \in
      \D^{\hat{\gamma}_{q+1}}} \biggl\{\, x \in I_{\lambda}: \\ :
    \biggl| x-\sum_{i=1}^{\hat{\gamma}_{q+1}} \omega_i \lambda^{i}
    \biggr| < 2^{-\hat{\gamma}_{q+1}\alpha} \,\biggr\}+\Z
    \diam(I_{\lambda})\Biggr).
  \end{multline}
  In this way we have defined two families of closed intervals
  $(\Delta_{\omega})_{\omega\in\Sigma_*}$ and
  $(\hat{\Delta}_{\omega})_{\omega\in\Sigma_*}$ with the
  property that for any $\omega \in \Sigma_*$ and $\eta \in \D$,
  \begin{equation*}
    \Delta_{\omega}\supseteq \hat{\Delta}_{\omega}\supseteq
    \Delta_{\omega\eta},
  \end{equation*}
  and given any $\omega \in \D^n$, $\diam(\Delta_{\omega})\leq
  \lambda^{n+1}/(1-\lambda)$. Thus, we may define a map $\pi:\Sigma
  \rightarrow I_{\lambda}$ by
  \begin{equation*}
    \pi(\omega):=\bigcap_{n\in \N}\Delta_{\omega|n}=\bigcap_{n\in
      \N}\hat{\Delta}_{\omega|n}.
  \end{equation*}
  By construction we also have $\delta_n\geq \lambda^2/2
  \hat{\delta}_{n-1}$ for all $n\in \N$.

  We let $\Lambda:=\pi\left(\Sigma\right)$. By Equations (\ref{Good
    approximant prop}) and (\ref{Infinite k}) we have
  \begin{equation*}
    \Lambda \subset \bigcap_{j\in\N} f_j\left(W_{\lambda}(\alpha)+\Z \cdot \diam(I_{\lambda})\right).
  \end{equation*}

  Thus, to complete the proof it suffices to show that $\dimH \Lambda
  \geq s$. In order to do this we shall define a measure supported on
  $\Lambda$ with the property
  \begin{equation*}
    \Cdim(\nu):=\liminf_{r\rightarrow 0}\frac{1}{\log r}\log \int
    \nu\left(B_r(x)\right) d\nu(x) \geq s.
  \end{equation*}
  That is, the correlation dimension $\Cdim(\nu)$ of $\nu$ is at least
  $s$. This implies that the Hausdorff dimension of $\nu$ and hence
  $X$ is at least $s$ (see \cite[Section 17]{Pesin}).
	
  We do this by taking $\mu$ to be the
  $\left(\frac{1}{2},\frac{1}{2}\right)$-Bernoulli measure on $\Sigma$
  and $\nu$ its projection by $\pi$, $\nu:=\mu\circ \pi^{-1}$.

  In order to estimate $\Cdim(\nu)$ we require good upper bounds on
  the number of intervals $\hat{\Delta}_{\omega}$ of a given level
  which are close to one another.

  \begin{lemma}\label{first estimate for t close}
    Suppose $\rho>1$ and $\lambda \notin A(s)$. Then there exists a
    constant $C$ depending only on $\rho$ and $\lambda$ such that for
    any pair $\eta,\zeta \in \D^{\Gamma (q)}$ for some
    $q\in \N$, $n=l + \Gamma (q)$ for some $l \leq
    \gamma_{q+1}$ and $\hat{\delta}_n \leq t \leq \rho \cdot \delta_n$
    we have,
    \begin{equation*}
      \#\left\lbrace (\kappa,\tau)\in \D^l : d
      (\hat{\Delta}_{\eta\kappa}, \hat{\Delta}_{\zeta\tau} ) <
      t\right\rbrace \leq C \cdot 2^l+4^{l}l \cdot(\rho
      \lambda)^{-s}\left(\frac{t}{\hat{\delta}_{n-l}}\right)^s.
    \end{equation*}
  \end{lemma}

  \begin{proof}
    We begin by noting that for each pair $(\kappa,\tau)\in \D^l$ we
    have
    \begin{eqnarray*}
      f_{j(q)}\circ g_{\omega(\eta)}\circ
      g_{\kappa}(0)+f_{j(q)}(n(\eta)) &\in &
      \hat{\Delta}_{\eta\kappa}\\ f_{j(q)}\circ g_{\omega(\zeta)}\circ
      g_{\tau}(0)+f_{j(q)}(n(\zeta)) &\in & \hat{\Delta}_{\zeta\tau}.
    \end{eqnarray*} 
    Since every $\hat{\Delta}_{\eta\kappa}, \hat{\Delta}_{\zeta\tau}$
    has diameter $\hat{\delta}_n$, $t\geq \hat{\delta}_n$ we have,
    \begin{multline*}
      \# \bigl\lbrace\, (\kappa,\tau)\in \D^l : d
      (\hat{\Delta}_{\eta\kappa}, \hat{\Delta}_{\zeta\tau} ) < t
      \,\bigr\rbrace \\ \leq \#\Bigl\{\, (\kappa,\tau)\in \D^l:\big|
      f_{j(q)}\circ g_{\omega(\eta)}\circ g_{\kappa}(0)- f_{j(q)}\circ
      g_{\omega(\zeta)}\circ g_{\tau}(0) + \\ +
      (f_{j(q)}(n(\eta))-f_{j(q)}(n(\zeta)))\big|< 2t \,\Bigr\}
    \end{multline*}
    Since $f_{j(q)}$ is affine we have,
    \begin{multline*}
      f_{j(q)}\circ g_{\omega(\zeta)}\circ g_{\tau}(0) \\ = \left(
      f_{j(q)}\circ g_{\omega(\eta)}\circ g_{\tau}(0)+\left(
      f_{j(q)}\circ g_{\omega(\zeta)}(0)- f_{j(q)}\circ
      g_{\omega(\eta)}(0)\right)\right).
    \end{multline*}
    By applying Lemma \ref{translation lemma} we obtain
    \begin{align*}
      &\#\Bigl\{\, (\kappa,\tau)\in \D^l:\big| f_{j(q)}\circ
      g_{\omega(\eta)}\circ g_{\kappa}(0)- f_{j(q)}\circ
      g_{\omega(\zeta)}\circ g_{\tau}(0) + \\ & \hspace{6cm}
      +(f_{j(q)}(n(\eta))-f_{j(q)}(n(\zeta)))\big|< 2t
      \,\Bigr\} \\ &\leq 4 \#\left\lbrace\, (\kappa,\tau)\in
      \D^l:\big| f_{j(q)}\circ g_{\omega(\eta)}\circ g_{\kappa}(0)-
      f_{j(q)}\circ g_{\omega(\eta)}\circ g_{\tau}(0)\big|< 2t
      \,\right\rbrace.
    \end{align*}
    We note that
    \begin{eqnarray*}
      \lVert (f_{j(q)}\circ g_{\omega(\eta)})' \rVert_{\infty}&\geq&
      \frac{1-\lambda}{\lambda}\cdot \diam\left(f_{j(q)}\circ
      g_{\omega(\kappa)}(I_{\lambda}) \right)\\ &\geq&
      \frac{1-\lambda}{\lambda}\cdot \frac{\lambda}{2}\cdot
      \diam (\hat{\Delta}_{\kappa} )\\ &=& (1-\lambda) \cdot
      \hat{\delta}_{n-l}/2.
    \end{eqnarray*}
    Piecing the above together we have
    \begin{align*}
      \# \bigl\{\, (\kappa,\tau) &\in \D^l:d
      (\hat{\Delta}_{\eta\kappa}, \hat{\Delta}_{\zeta\tau}) < t
      \,\bigr\} \\ &\leq 4\cdot \bigl\{\, (\kappa,\tau)\in \D^l :
      \big| f_{j(q)}\circ g_{\omega(\eta)}\circ g_{\kappa}(0)-
      \\ & \hspace{5.5cm} - f_{j(q)}\circ g_{\omega(\eta)}\circ
      g_{\tau}(0)\big|< 2t\,\bigr\} \\ &\leq 4\cdot \#\bigl\{
      (\kappa,\tau)\in \D^l:\big|g_{\kappa}(0)- g_{\tau}(0)\big|< 2t
      \lVert (f_{j(q)}\circ g_{\omega(\eta)})'\rVert_\infty^{-1}
      \bigr\rbrace \\&\leq 4\cdot \#\biggl\lbrace (\kappa,\tau)\in
      \D^l:\big|g_{\kappa}(0)-g_{\tau}(0)\big|<
      \frac{4t}{(1-\lambda)\hat{\delta}_{n-l}}\biggr\rbrace.
    \end{align*}
    Since $\delta_n \leq \hat{\delta}_{n-l}\cdot \lambda^l$ and $t \leq
    \rho \delta_n$ we have,
    \begin{eqnarray*}
      \frac{4t}{(1-\lambda)\hat{\delta}_{n-l}} &\leq& \frac{4
        \rho}{1-\lambda} \cdot \lambda^l.
    \end{eqnarray*}
    Now choose $k\in \N\cup\{0\}$ so that
    \begin{eqnarray*}
      \frac{4 \rho}{1-\lambda} \cdot
      \lambda^{l+k+1}<\frac{4t}{(1-\lambda)\hat{\delta}_{n-l}} \leq
      \frac{4 \rho}{1-\lambda} \cdot \lambda^{l+k}.
    \end{eqnarray*}
    By applying Lemma \ref{Constant depending on r} we have
    \begin{align*}
      \#\{\, (\kappa,\tau) &\in \D^l : d (\hat{\Delta}_{\eta\kappa},
      \hat{\Delta}_{\zeta\tau}) < t \,\} \\ &\leq 4\cdot \#\{\,
      (\kappa,\tau)\in \D^l:\big|g_{\kappa}(0)-g_{\tau}(0)\big|<4 \rho
      (1-\lambda)^{-1} \lambda^{l+k} \,\} \\ &= 4\cdot \tilde{P}_l
      \Bigl(\lambda,k,\frac{4 \rho}{1-\lambda} \Bigr)\\ &\leq C\Bigl(
      \frac{4 \rho}{1-\lambda}\Bigr) \cdot
      2^l+4^{l}l\lambda^{s(l+k)}\\ &\leq C\Bigl( \frac{4
        \rho}{1-\lambda}\Bigr) \cdot 2^l+4^{l}l \cdot(\rho
      \lambda)^{-s} \Bigl( \frac{t}{\hat{\delta}_{n-l}}
      \Bigr)^s. \qedhere
    \end{align*}
  \end{proof}
  
  \begin{lemma}\label{second estimate for r close}
    Suppose $\rho>1$ and $\lambda \notin A(s)$. Then there exists a
    constant $C$ depending only on $\rho$ and $\lambda$ such that
    given $q\in \N$ and $n=l + \Gamma (q)$ for some $l \leq
    \gamma_{q+1}$ and $\hat{\delta}_n \leq t \leq \rho \cdot \delta_n$
    we have,
    \begin{multline*}
      \# \{\, (\kappa,\tau)\in \D^n : d (\hat{\Delta}_{\kappa},
      \hat{\Delta}_{\tau}) < t \,\} \\ \leq
      4^{\Gamma (q-1)} \cdot \biggl( C \cdot
      2^{\gamma_q}+4^{\gamma_q}\gamma_q \cdot \lambda^{-s} \biggl(
      \frac{\hat{\delta}_{n-l}}{\hat{\delta}_{n-l-\gamma_q}} \biggr)^s
      \biggr)\\ \cdot \left(C \cdot 2^l+4^{l}l \cdot(\rho
      \lambda)^{-s}\left(\frac{t}{\hat{\delta}_{n-l}}\right)^s\right).
    \end{multline*}
  \end{lemma}
  
  \begin{proof}
    First note that if $\eta \in \D^{n-l}$ and $\alpha \in \D^l$,
    $\hat{\Delta}_{\eta\alpha}\subseteq \hat{\Delta}_{\eta}$. Hence,
    \begin{align*}
      & \#\left\lbrace (\kappa,\tau)\in \D^n:d (\hat{\Delta}_{\kappa},
      \hat{\Delta}_{\tau}) < t\right\rbrace \\ &=
      \sum_{(\kappa,\tau)\in \D^n}\chi_{\left\lbrace
        (\kappa',\tau')\in \D^n: d\left(\hat{\Delta}_{\kappa'},
        \hat{\Delta}_{\tau'}\right)< t\right\rbrace} \\ &=
      \sum_{(\eta,\zeta)\in \D^{n-l}}\chi_{\{ (\eta',\zeta')\in
        \D^{n-l}: d (\hat{\Delta}_{\eta'}, \hat{\Delta}_{\zeta'} )< t
        \}}\cdot \sum_{(\alpha,\beta)\in \D^l} \chi_{\{
        (\alpha',\beta')\in \D^{l}: d (\hat{\Delta}_{\eta\alpha'},
        \hat{\Delta}_{\zeta\beta'} )< t \}}\\ &= \sum_{(\eta,\zeta)\in
        \D^{n-l}}\chi_{\{ (\eta',\zeta')\in \D^{n-l}: d
        (\hat{\Delta}_{\eta'}, \hat{\Delta}_{\zeta'} ) < t \}} \cdot
      \\ & \hspace{4cm} \cdot \#\bigl\{ (\alpha,\beta)\in \D^l: d
      (\hat{\Delta}_{\eta\alpha}, \hat{\Delta}_{\zeta\beta}) < t \}.
    \end{align*}
    By applying Lemma \ref{first estimate for t close} along with the
    fact that $t\leq \rho \delta_{n}\leq \rho \hat{\delta}_{n-l}$,
    \begin{align*}
      &\# \left\lbrace (\kappa,\tau)\in \D^n:d (\hat{\Delta}_{\kappa},
      \hat{\Delta}_{\tau} )< t\right\rbrace \\ &\leq \#\left\lbrace
      (\eta,\zeta) \in \D^{n-l}:d
      (\hat{\Delta}_{\eta},\hat{\zeta}_{\tau}) < t \right\rbrace \cdot
      \left(C 2^l+4^{l}l \cdot(\rho
      \lambda)^{-s}\left(\frac{t}{\hat{\delta}_{n-l}}\right)^s\right)\\ &\leq
      \#\left\lbrace (\eta,\zeta) \in \D^{n-l}:d
      (\hat{\Delta}_{\eta},\hat{\zeta}_{\tau}) < \rho
      \hat{\delta}_{n-l} \right\rbrace \cdot \\ & \hspace{4cm} \cdot
      \left(C 2^l+4^{l}l \cdot(\rho
      \lambda)^{-s}\left(\frac{t}{\hat{\delta}_{n-l}}\right)^s\right),
    \end{align*}
    Now clearly $\rho \hat{\delta}_{n-l} \in [\hat{\delta}_{n-l},\rho
      \hat{\delta}_{n-l}]$ and so we may apply the above reasoing to
    the first term to obtain,
    \begin{align*}
      &\# \left\lbrace (\eta,\zeta) \in \D^{n-l}:d
      (\hat{\Delta}_{\eta},\hat{\zeta}_{\tau}) < \rho
      \hat{\delta}_{n-l} \right\rbrace \\ &\leq \#\left\lbrace
      (\alpha,\beta) \in \D^{n-l-\gamma_q}:d
      (\hat{\Delta}_{\alpha},\hat{\beta}_{\tau}) < \rho
      \hat{\delta}_{n-l-\gamma_q} \right\rbrace \cdot
      \\ & \hspace{4cm} \cdot \biggl( C \cdot
      2^{\gamma_q}+4^{\gamma_q}\gamma_q \cdot \lambda^{-s} \biggl(
      \frac{\hat{\delta}_{n-l}}{\hat{\delta}_{n-l-\gamma_q}} \biggr)^s
      \biggr)\\ &\leq \#\D^{2\sum_{p<q}\gamma_p} \cdot \biggl( C \cdot
      2^{\gamma_q}+4^{\gamma_q}\gamma_q \cdot \lambda^{-s} \biggl(
      \frac{\hat{\delta}_{n-l}}{\hat{\delta}_{n-l-\gamma_q}} \biggr)^s
      \biggr).
    \end{align*}
    Piecing these two inequalities together completes the proof of the
    lem\-ma.
  \end{proof}

  Recall that to complete the proof we must obtain the following
  inequality,
  \begin{equation*}
    \Cdim(\nu)=\liminf_{r\rightarrow 0}\frac{1}{\log r}\log \int
    \nu\left(B_r(x)\right) d\nu(x) \geq s.
  \end{equation*}
  Choose $r\in \left(0,\lambda/(1-\lambda)\right)$ and take $n$ to be
  the least integer satisfying $\hat{\delta}_n<r$. It follows that $r
  \leq \hat{\delta}_{n-1}< 2/\lambda^2 \cdot \delta_n$. Given
  $\kappa\in \D^n$ and a sequence $\omega$ such that $\kappa = \omega
  | n$, we have
  \begin{equation*}
    \#\lbrace\, \tau \in \D^n : \hat{\Delta}_{\tau}\cap
    B_r(\pi(\omega))\neq \emptyset \,\rbrace \leq \#\lbrace\, \tau \in
    \D^n : d (\hat{\Delta}_{\tau},\hat{\Delta}_{\kappa}) < r
    \,\rbrace.
  \end{equation*}
  Hence,
  \begin{eqnarray*}
    \nu\left(B(\pi(\omega),r)\right) \leq \#\lbrace\, \tau \in \D^n
    : d (\hat{\Delta}_{\tau},\hat{\Delta}_{\kappa}) < r \,\rbrace
    \cdot 2^{-n}.
  \end{eqnarray*}
  Since $\nu=\mu\circ \pi^{-1}$ we have,
  \begin{align*}
    \int \nu\left(B_r(x)\right) & d\nu(x) =\int
    \nu\left(B_r(\pi(\omega))\right) d\mu(\omega)\\ &\leq
    \sum_{\kappa\in \D^n} \mu([\kappa]) \bigl(\# \lbrace\, \tau \in
    \D^n : d (\hat{\Delta}_{\tau},\hat{\Delta}_{\kappa}) < r \,\rbrace
    \cdot 2^{-n}\bigr)\\ &= 4^{-n} \# \lbrace\, (\kappa,\tau) \in
    (\D^n)^2 : d (\hat{\Delta}_{\tau},\hat{\Delta}_{\kappa}) < r
    \,\rbrace.
  \end{align*}
  Now note that $\hat{\delta}_n< r\leq 2/\lambda^2 \delta_n \leq 8
  \delta_n$ so by Lemma \ref{second estimate for r close} we have,
  \begin{align} \label{integral key estimate with r involved}
    \int & \nu\left(B_r(x)\right) d\nu(x) \\ &\leq 4^{-n} \cdot
    4^{\Gamma (q-1)} \cdot \biggl( C \cdot
    2^{\gamma_q}+4^{\gamma_q}\gamma_q \cdot \lambda^{-s} \biggl(
    \frac{\hat{\delta}_{n-l}}{\hat{\delta}_{n-l-\gamma_q}} \biggr)^s
    \biggr) \cdot \nonumber \\ & \hspace{4cm} \cdot \left(C \cdot
    2^l+4^{l}l \lambda^{-s} \left( \frac{r}{\hat{\delta}_{n-l}}
    \right)^s \right) \nonumber \\ &\leq \biggl(C \cdot
    2^{-\gamma_q}+\gamma_q \cdot \lambda^{-s}
    \biggl(\frac{\hat{\delta}_{n-l}}{\hat{\delta}_{n-l-\gamma_q}}
    \biggr)^s \biggr) \cdot \nonumber \\ & \hspace{4cm} \cdot \left(C
    \cdot 2^{-l}+l \lambda^{-s} \left( \frac{r}{\hat{\delta}_{n-l}}
    \right)^s \right). \nonumber
  \end{align}
  where $q$ is chosen so that $n=l + \Gamma (q)$ and
  $0\leq l<\gamma_{q+1}$.

  Now since $m_q \geq \left(\frac{\log 2^{-\alpha}}{\log
    \lambda}-1\right)\gamma_q$,
  \begin{eqnarray*}
    \frac{\hat{\delta}_{n-l}}{\hat{\delta}_{n-l-\gamma_q}} \leq
    \lambda^{\gamma_q+m_q} \leq 2^{-\alpha \gamma_q},
  \end{eqnarray*}
  and provided $l>0$ we have
  \begin{eqnarray*}
    \frac{r}{\hat{\delta}_{n-l}} \leq
    \frac{8\delta_n}{\hat{\delta}_{n-l}} \leq \lambda^l.
  \end{eqnarray*}
  Note that $\frac{1}{2}\leq \lambda$ and since $s\leq
  \frac{1}{\alpha}$, we have $2^{-\gamma_q}\leq
  \left(2^{-\alpha\gamma_q}\right)^s$. Thus, by
  Equation~(\ref{integral key estimate with r involved}), if $l>0$ we
  have
  \begin{eqnarray} \label{integral key estimate with r removed l>0}
    \int \nu\left(B_r(x)\right) d\nu(x) &\leq & (2 C \lambda^{-s})^2
    \cdot \gamma_q \left(2^{-\alpha \gamma_q}\right)^s \cdot l
    \left(\lambda^l \right)^s,
  \end{eqnarray}
  and if $l=0$ we have,
  \begin{eqnarray} \label{integral key estimate with r removed l=0}
    \int \nu\left(B_r(x)\right) d\nu(x) &\leq & (2 C^2 \lambda^{-s})
    \cdot \gamma_q \left(2^{-\alpha \gamma_q}\right)^s.
  \end{eqnarray}

  By the inequality (\ref{m q inequalities}) we have,
  \begin{eqnarray}\label{r is not too small}
    r&>& \hat{\delta}_n \geq \hat{\delta}_{\Gamma (q)}
    \cdot \frac{ \lambda}{2} \cdot \lambda^l\\ \nonumber &\geq
    &\hat{\delta}_{\Gamma (q-1)} \left(\frac{
      \lambda}{2}\right)^2 \cdot \lambda^{\gamma_q + m_q} \cdot
    \lambda^l\\ \nonumber &\geq &\hat{\delta}_{\Gamma (q-1)}
    \left(\frac{ \lambda}{2}\right)^2 \cdot
    \lambda^{\hat{\gamma}_q+m_q} \cdot \lambda^l\\ \nonumber &\geq &
    \frac{\lambda^3}{4}\cdot\hat{\delta}_{\Gamma (q-1)} \cdot
    2^{-\alpha\gamma_q -\alpha \theta_q} \cdot \lambda^{l}.
  \end{eqnarray}
  Now by construction, for each $q\in \N$, $\gamma_{q}>q \log (
  \hat{\delta}_{\Gamma(q-1)} )^{-1} \cdot \theta_q$, so if we define
  \begin{equation*}
    \iota(q):= \frac{-\log \left(\lambda^3/ \log 4\right) - \log
      \hat{\delta}_{\Gamma (q-1)}+\theta_q\alpha \log
      2}{\gamma_q\alpha \log 2},
  \end{equation*}
  we have $\iota(q) \rightarrow 0$ as $q \rightarrow
  \infty$. Moreover, by (\ref{r is not too small}),
  \begin{equation*}
    \frac{\gamma_q\log{2}+ l\log \lambda^{-1}}{-\log r} \geq
    \frac{1}{1+\iota(q)}.
  \end{equation*}
  Substituting into Equations (\ref{integral key estimate with r
    removed l>0}) and (\ref{integral key estimate with r removed l=0})
  and noting that $q\rightarrow \infty$ as $r\rightarrow 0$ we have,
  \begin{equation*}
    \Cdim(\nu)=\liminf_{r\rightarrow 0}\frac{1}{\log r}\log \int
    \nu\left(B_r(x)\right) d\nu(x) \geq s.
  \end{equation*}
  This completes the proof of the Proposition.
\end{proof}

\section{$\beta$-shifts and a uniform lower bound} \label{sec:lowerbound}

Let $1 < \beta \leq 2$. Given a real number $x \in \R$ we let $\lfloor
x \rfloor$ and $\{x\}$ denote, respectively, the integer and
fractional parts of $x$. Consider the $\beta$-transformation $f_\beta
\colon [0,1) \to [0,1)$ defined by $x \mapsto \{\beta x \}$. Given $x
  \in [0,1]$ we let $\omega^{\beta}_n(x):=\lfloor \beta
  f_{\beta}^{n-1}(x) \rfloor$ and
\begin{equation*}
  S_{\beta}:=\closure \bigl\{\, (\omega^{\beta}_n(x))_{n \in \N}:
  x \in [0,1) \,\bigr\}.
\end{equation*}
Let $\pi_{\beta} \colon S_{\beta} \rightarrow [0,1]$ be defined by
$(\omega_n)_{n \in \N} \mapsto \sum_{n \in \N} \omega_n \beta^{-n}$,
and let $\sigma \colon S_{\beta} \rightarrow S_{\beta}$ denote the
left shift operator on $S_{\beta}$. Note that $\pi_{\beta} \circ
\sigma= f_{\beta}\circ \pi_{\beta}$. Parry proved in \cite{parry} that
the shift space $S_\beta$ can be written as
\[
S_\beta = \{\, (\omega_1, \omega_2, \ldots) \in \{0,1\}^\N : \sigma^k
(\omega_1, \omega_2, \ldots) \leq (\omega_n^\beta(1^-))_{n \in \N}
\ \forall k \,\},
\]
where $\leq$ is the lexicographical order and $\omega_n^\beta (1^-)$
denotes the limit in the product topology of $\omega_n^\beta (x)$ as
$x \to 1$. Moreover, Parry proved that $S_{\beta}$ is a subshift of
finite type if and only if the sequence $(\omega_n^\beta (1))_{n \in
  \N}$ terminates with infinitely many zeroes, and that a sequence
$(\omega_n)_{n\in\N}$ equals $(\omega_n^\beta (1))_{n\in\N}$
for some $\beta$ if and only if it satisfies
\begin{equation} \label{eq:developmentsof1}
(\omega_k, \omega_{k+1}, \ldots) < (\omega_1, \omega_2, \ldots)
\end{equation}
for all $k > 1$.  In the set of sequences satisfying
\eqref{eq:developmentsof1}, the subset of sequences terminating with
infinitely many zeroes is dense. This implies that the set of $\beta$
for which the sequence $(\omega_n^\beta(1))_{n \in \N}$ terminates
with infinitely many zeroes is dense in $(1,2)$. Hence $S_\beta$ is a
subshift of finite type for a dense set of $\beta$.

The following theorem allows us to transfer results from subshifts of
finite type to arbitrary $\beta$-shifts. It is a strengthened version
of Theorem~2 from \cite{farmperssonschmeling}, that follows immediately
by replacing Lemma~6 in \cite{farmperssonschmeling}, by Lemma~1 in
\cite{farmpersson}.

\begin{theorem}[F\"{a}rm, Persson]\label{Farm Persson finite type theorem}
  Let $\beta \in (1,2)$ and let $(\beta_n)_{n \in \N}$ be any sequence
  with $1<\beta_n < \beta$ for all $n$, such that $\beta_n \rightarrow
  \beta$ as $n \rightarrow \infty$. Suppose $E \subset S_{\beta}$ and
  $\pi_{\beta_n}\left(E \cap S_{\beta_n}\right)$ is in the class
  $\mathcal{G}^s (I)$ for all $n$. If $F$ is a $G_{\delta}$ with $F
  \supset \pi_{\beta}\left(E \cap S_{\beta}\right)$, then $F$ is also
  in the class $\mathcal{G}^s (I)$.
\end{theorem}
For $\kappa > 0$, we consider the sets
\[
A_\beta (\kappa) = \{\, x \in [0,1] : 0 \leq T_\beta^n (x) \leq
\beta^{-\kappa n} \text{ infinitely often} \, \}.
\]
We shall use the following theorem which allows us to restrict our attention to the case where $S_{\beta}$ is a subshift of finite type.
\begin{theorem} \label{the:beta}
  For any $1 < \beta \leq 2$ we have $A_\beta (\kappa) \in
  \mathcal{G}^s ([0,1])$ for $s = \frac{1}{1+\kappa}$.
\end{theorem}

\begin{remark}
  We note that the bound $s \leq \frac{1}{1+\kappa}$ is sharp since an
  easy covering argument, using the fact that $T_{\beta}$ has
  topological entropy $\log \beta$, shows that the Hausdorff dimension
  of $A_\beta (\kappa)$ is not larger than $\frac{1}{1+\kappa}$.
\end{remark}

\begin{proof}
  We let
  \[
  A_{\beta,n} (\kappa) = \biggl\{ x : 0 \leq x - y \leq 2^{- \gamma n}
  \text{ for some } y = \sum_{k=1}^n \frac{a_k}{\beta^k},\ (a_k)_{k
    \in \N} \in S_\beta \, \biggr\},
  \]
  and note that $A_\beta (\kappa)$ can be written as $A_\beta (\kappa)
  = \limsup_{n\to \infty} A_{\beta,n} (\kappa)$.

  By Theorem \ref{Farm Persson finite type theorem} it suffices to
  prove the theorem in the special case where $S_\beta$ is a subshift
  of finite type.

  When $S_\beta$ is a subshift of finite type there are constants
  $c_1$ and $c_2$ such that
  \begin{equation} \label{eq:cylinderlength}
    c_1 \beta^{-n} \leq | \pi_\beta ([a_1, a_2, \ldots, a_n]) | \leq
    c_2 \beta^{-n}.
  \end{equation}
  This implies that the number of cylinders of size $n$, denoted by $N
  (n)$, satisfies
  \begin{equation} \label{eq:numberofcylinders}
    c_2^{-1} \beta^n \leq N(n) \leq c_1^{-1} \beta^n.
  \end{equation}
  Using these estimates we may complete the proof by following the
  method of \cite[Example~8.9]{falconerbook}.
\end{proof}

\begin{corollary}\label{cor:lowerbound}
  For any $\lambda \in \left(\frac{1}{2},1\right)$ and $\alpha>1$ we
  have $W_{\lambda}(\alpha) \in \mathcal{G}^s (I_\lambda)$ for $s =
  \frac{-\log \lambda}{\alpha \log 2}$.
\end{corollary}

\begin{proof}
  Take $\beta =\lambda^{-1}$ and $\kappa = \frac{\alpha \log 2}{\log
    \beta}-1$. It follows that $A_{\beta}(\kappa) \subset
  W_{\lambda}(\alpha)$, so $W_\lambda (\alpha) \in \mathcal{G}^s
  ([0,1])$ follows immediately from Theorem \ref{the:beta}. Now, the
  self-similar structure of $W_\lambda (\alpha)$ implies that
  $W_\lambda (\alpha) \in \mathcal{G}^s (I_\lambda)$.
\end{proof}

\section{Covering arguments and upper bounds}\label{s covering args}

Each of the upper bounds from Theorem \ref{main bullet points} parts
(1), (3) and (5) will rely on the following simple relationship
between the growth in the number of $n$th level $\lambda$ sums and the
dimension of $W_{\lambda}(\alpha)$. Given $\lambda \in
\left(\frac{1}{2},1\right)$ and $n \in \N$ we let
\begin{equation*}
  F_{\lambda, n} := \biggl\{\, \sum_{k=1}^n a_k \lambda^k : a_k \in
  \{0,1\} \,\biggr\},
\end{equation*}
and let 
\begin{equation*}
  \tau(\lambda):= \limsup_{n \rightarrow \infty} \frac{\log
    \#F_{\lambda,n}}{n \log 2}.
\end{equation*}

\begin{lemma}\label{tau cover lemma}
  For all $\lambda \in \left(\frac{1}{2},1\right)$ and $\alpha >1$ the
  Hausdorff dimension of $W_{\lambda}(\alpha)$ is bounded above by
  $\tau(\lambda)/\alpha$.
\end{lemma}
\begin{proof}
  This may be deduced by a standard covering argument. See for example
  the first paragraph in the proof of Jarn\'{i}k's theorem from
  \cite[Section 10.3]{falconerbook}.
\end{proof}

Our first corollary establishes Theorem \ref{main bullet points} (1).

\begin{corollary}
  For all $\lambda \in \left(\frac{1}{2},1\right)$ and $\alpha >1$ the
  Hausdorff dimension of $W_{\lambda}(\alpha)$ is bounded above by
  $1/\alpha$.
\end{corollary}

\begin{proof}
  This is immediate from Lemma \ref{tau cover lemma} combined with the
  fact that $\#F_{\lambda,n} \leq 2^n$ so $\tau(\lambda) \leq 1$ for
  all $\lambda \in \left(\frac{1}{2},1\right)$.
\end{proof}

Our second corollary establishes Theorem \ref{main bullet points} (3).

\begin{corollary}
  There exists a dense family $\Gamma \subset
  \left(\frac{1}{2},1\right)$ such that for all $\lambda \in \Gamma$,
  $\dim W_{\lambda}(\alpha)<1/\alpha$.
\end{corollary}

\begin{proof}
  Our approach is based on \cite{ss}. We let $\Gamma$ denote the set
  of $\lambda \in \left(\frac{1}{2},1\right)$ such that for some
  finite word $\left(\omega_i\right)_{i=1}^{n} \in \{0,1\}^n$ we have
  $1=\sum_{i=1}^n\omega_i \lambda^i$. To see that $\Gamma$ is dense in
  $\left(\frac{1}{2},1\right)$ first fix $\lambda_0 \in
  \left(\frac{1}{2},1\right)$ and $\epsilon\in (0,1-\lambda_0)$. Then
  there exists an infinite string
  $\left(\omega_i\right)_{i=1}^{\infty} \in \{0,1\}^{\N}$ with
  $1=\sum_{i=1}^{\infty}\omega_i \lambda_0^i$. Let $k$ be the smallest
  $q$ with $\omega_q =1$ and choose $n$ so that
  $\sum_{i=1}^{n}\omega_i \lambda_0^i>1-\epsilon^k$. Then for some
  $\lambda \in \left(\lambda_0, \lambda_0+\epsilon\right)$ we have
  $\sum_{i=1}^{n}\omega_i \lambda^i=1$, so $\left(\lambda_0,
  \lambda_0+\epsilon\right)\cap \Gamma \neq \emptyset$.

  By Lemma \ref{tau cover lemma} it suffices to show $\tau(\lambda)<1$
  for all $\lambda \in \Gamma$. But if $\lambda \in \Gamma$ then for
  some finite word $\left(\omega_i\right)_{i=1}^{n} \in \{0,1\}^n$ we
  have $\lambda^{q(n+1)+1}=\sum_{i=1}^n\omega_i \lambda^{i+1+q(n+1)}$
  for all $q \in \N$. It follows that for all $q \in \N$,
  \begin{align*}
    F_{\lambda, q(n+1)}=\\
    =\biggl\{\, \textstyle \sum\limits_{i=0}^{q-1}
    \sum\limits_{j=1}^{n+1} & a_{i(n+1)+j}\lambda^{i(n+1)+j} :\\
    ( & a_{i(n+1)+1},\ldots, a_{i(n+1)+(n+1)} ) \in
    \{0,1\}^{n+1} \,\biggr\}\\ =\biggl\{\, \textstyle \sum\limits_{i=0}^{q-1}
    \sum\limits_{j=1}^{n+1} & a_{i(n+1)+j}\lambda^{i(n+1)+j} :\\
    ( & a_{i(n+1)+1},\ldots, a_{i(n+1)+(n+1)} ) \in
    \{0,1\}^{n+1}\backslash\{(1,0,\ldots,0)\} \,\biggr\}.
  \end{align*}
  Thus, for each $q$ we have 
  \begin{equation*}
    \#F_{q(n+1)} \leq \left(2^{n+1}-1\right)^q,
  \end{equation*}
  so for all $l \in \N$, 
  \begin{equation*}
    \#F_{l, \lambda} \leq \#F_{\lceil l/(n+1) \rceil (n+1)} \leq
    \left(2^{n+1}-1\right)^{\lceil l/(n+1) \rceil}.
  \end{equation*}
  Thus, $\tau(\lambda) \leq \log \left(2^{n+1}-1\right)/(n+1) \log 2<1$.
\end{proof}

Finally we complete the proof of Theorem \ref{main bullet points} (5).

\begin{defn}
  A \textit{multinacci number} is a postive real $\lambda$ which
  satisfies an equation of the form $\lambda^m+\cdots+\lambda=1$ for
  some $m \in \N$.
\end{defn}
We note that there are countably many multinacci numbers, all of which
are contained within the interval $\left( \frac{1}{2}, 1\right)$. The
largest multinacci number is the golden ratio $\frac{\sqrt{5}-1}{2}$.

\begin{theorem}
  Let $\lambda$ be a multinacci number. Then the Hausdorff dimension
  of $W_\lambda (\alpha)$ is $- \frac{\log \lambda}{\log 2}
  \frac{1}{\alpha}$.
\end{theorem}

\begin{proof}
  Put
  \begin{align*}
    S_1 \colon x &\mapsto \lambda x, \\
    S_2 \colon x &\mapsto \lambda (x + 1).
  \end{align*}
  Let us first consider the case $m = 2$. Then $\lambda =
  \frac{\sqrt{5} - 1}{2}$ and $S_1 \circ S_2 \circ S_2 = S_2 \circ S_1
  \circ S_1$. Hence, when defining $W_\lambda (\alpha)$ we need only
  consider sequences where the word $011$ is forbidden, since
  replacing the word $011$ in a sequence by the word $100$, yields the
  same point. Hence, if we put
  \[
  F_{\lambda, n} = \biggl\{\, \sum_{k=1}^n a_k \lambda^k : a_k \in
  \{0,1\} \,\biggr\},
  \]
  then we have
  \[
  F_{\lambda, n} = \biggl\{\, \sum_{k=1}^n a_k \lambda^k : a_k \in
  \{0,1\},\ (a_k, a_{k+1}, a_{k+2}) \neq (0,1,1) \,\biggr\}.
  \]

  The subshift in which $011$ is forbidden is a subshift of finite
  type, with adjacency matrix
  \[
  A = \left[ \begin{array}{cccc} 1 & 1 & 0 & 0 \\ 0 & 0 & 1 & 0 \\ 1 &
      1 & 0 & 0 \\ 0 & 0 & 1 & 1 \end{array} \right].
  \]
  One checks that $\lambda^{-1} = \frac{\sqrt{5}+1}{2}$ is the largest
  eigenvalue of $A$. Hence there is a constant $K$ such that $\#
  F_{\lambda, n} < K \lambda^{-n}$. Hence $\tau(\lambda)= -\frac{\log
    \lambda}{\log 2}$, so by Lemma \ref{tau cover lemma} the Hausdorff
  dimension of $W_{\lambda} (\alpha)$ is at most $- \frac{\log
    \lambda}{\log 2} \frac{1}{\alpha}$. But by
  Corollary~\ref{cor:lowerbound} it is at least $- \frac{\log
    \lambda}{\log 2} \frac{1}{\alpha}$.

  For a general $m \geq 2$ we proceed similarly. Assume $\lambda$ is
  such that $S_1 \circ S_2^m = S_2 \circ S_1^m$. This implies that
  $\lambda$ satisfies the equation
  \begin{equation} \label{eq:multinacci}
    \lambda^m + \lambda^{m-1} + \cdots + \lambda = 1.
  \end{equation}
  As before, for the set $F_{\lambda, n}$, we need only consider
  sequences where the word
  \[
  0 \underbrace{11 \ldots 1}_{m}
  \]
  is forbidden. This is again a subshift of finite type, and it can
  be represented using a $2^m \times 2^m$ adjacency matrix given by
  \[
  A = \left[ \begin{array}{ccccccc} 1 & 1 & & \\ & & 1 & 1 \\ & & & &
      \ddots \\ & & & & & 1 & 0 \\ 1 & 1 & & \\ & & 1 & 1 \\ & & & &
      \ddots \\ & & & & & 1 & 1 \end{array} \right].
  \]
  By the Perron--Frobenius theorem, the eigenvalue of largest modulus
  of this matrix, is a positive number, and it has a corresponding
  eigenvalue with positive elements. Let $v =
  [v_1\ \cdots\ v_{2^m}]^\mathrm{T}$ be such an eigenvector and let
  $\mu$ be the eigenvalue. It is not hard to see that the equation $A
  v = \mu v$ implies that
  \[
  \begin{array}{rcl} v_1 & = & v_{2^{m-1}+1}, \\
    v_2 & = & v_{2^{m-1}+2}, \\ & \vdots & \\ v_{2^{m-1}-1} & = &
    v_{2^m - 1}. \end{array}
  \]
  Let $1 \leq k < 2^{m-2}-1$. Looking at row $k$ and row $2^{m-2}+k$
  in the equation $Av = \mu v$, we see that $v_k =
  v_{2^{m-2}+k}$. Continuing in this fashion we end up in the
  conclusion that all $v_k$ for odd $k$ are equal. Without loss of
  generality we can therefore assume that $v_k = 1$ for odd $k$.

  If we look at the first row of the matrixes in the equation $A v =
  \mu v$, we see that $\mu v_1 = v_1 + v_2 = 1 + v_2$. We continue,
  and looking at the second row, we see that $\mu v_2 = v_3 + v_4 = 1
  + v_4$. Hence we have
  \[
  \mu = \mu v_1 = 1 + v_2 = 1 + \mu^{-1} (1 + v_4).
  \]
  Similarly we get $\mu v_4 = 1 + v_8$, and so
  \[
  \mu = 1 + \mu^{-1} + \mu^{-2} (1 + v_8).
  \]
  We can continue this process, using the equations
  \[
  \mu v_{2^k} = 1 + v_{2^{k+1}},
  \]
  that are valid for $0 \leq k \leq m-3$, to conclude
  \[
  \mu = 1 + \mu^{-1} + \cdots + \mu^{-m+2} (1 +
  v_{2^{m-2}}).
  \]
  But we have $\mu v_{2^{m-2}} = v_{2^{m-1}-1} = 1$, hence
  \[
  \mu = 1 + \mu^{-1} + \cdots + \mu^{-m+2} +
  \mu^{-m+1},
  \]
  or equivalently
  \[
  \mu^m = 1 + \mu + \cdots + \mu^{m-1}.
  \]
  Comparing with the equation \eqref{eq:multinacci}, this implies
  that we have $\mu = \lambda^{-1}$.
  
  The rest is just as for the case $m = 2$ above. We have that $\#
  F_{\lambda, n} < K \mu^n = K \lambda^{-n}$, and therefore
  $\tau(\lambda)= -\frac{\log \lambda}{\log 2}$, so by Lemma \ref{tau
    cover lemma} the Hausdorff dimension of $W_{\lambda} (\alpha)$ is
  at most $- \frac{\log \lambda}{\log 2} \frac{1}{\alpha}$.
\end{proof}

\end{document}